\documentclass[12pt,leqno]{amsart}\usepackage[latin9]{inputenc}\usepackage{verbatim,amsthm,amstext,amssymb,color,amscd,bm}\makeatletter
\usepackage{hyperref}\hypersetup{colorlinks=true,citecolor=black,linkcolor=black}\usepackage[shortlabels]{enumitem}\numberwithin{equation}{section}
\theoremstyle{plain}\newtheorem{theorem}{Theorem}[section]\newtheorem{lemma}[theorem]{Lemma}
\newtheorem{conjecture}[theorem]{Conjecture}\theoremstyle{definition}\newtheorem{remark}[theorem]{Remark}\makeatother
\begin{document}

\title[Infinite topological entropy and positive mean dimension]{Infinite topological entropy, positive mean dimension, and factors of subshifts}
\author[Lei Jin]{Lei Jin}
\address{Lei Jin: School of Mathematics, Sun Yat-sen University, Guangzhou, China}
\email{jinleim@mail.ustc.edu.cn}
\author[Yixiao Qiao]{Yixiao Qiao}
\address{Yixiao Qiao (Corresponding author): School of Mathematics and Statistics, HNP-LAMA, Central South University, Changsha, China}
\email{yxqiao@mail.ustc.edu.cn}
\thanks{L. Jin was supported by NNSF of China (Grant No. 12201653). Y. Qiao was supported by NNSF of China (Grant No. 12371190).}
\begin{abstract}
We study dynamical systems with the property that all the nontrivial factors have infinite topological entropy (or, positive mean dimension). We establish an ``if and only if'' condition for this property among a typical class of dynamical systems, the subshifts of block type in the Hilbert cube. This in particular leads to a large class of concrete (and new) examples of dynamical systems having this property.
\end{abstract}
\maketitle

\section{Introduction}
Let $X$ be a compact metrizable space and $T:X\to X$ a homeomorphism. We call the pair $(X,T)$ a (topological) \textbf{dynamical system}. A dynamical system $(Y,S)$ is called a \textbf{factor} of the system $(X,T)$ if there is a continuous and surjective map $\pi:X\to Y$ (called a \textbf{factor map}) satisfying that $\pi\circ T=S\circ\pi$.

Lindenstrauss \cite{Lin95} proved that\footnote{Actually what was proven there is much stronger than the statement here. But we use this for simplicity, as they have the same spirit. Notice that so will we do with \cite{Lin99}.} any finite-dimensional dynamical system has a nontrivial factor with finite topological entropy. (Here a dynamical system $(X,T)$ is \textbf{nontrivial} if the space $X$ contains at least two points, and is said to be finite-dimensional if the topological dimension, i.e. the Lebesgue covering dimension, of $X$ is finite.) However, for infinite-dimensional systems the case becomes quite different and is much more involved. On the one hand, Lindenstrauss \cite{Lin99} studied a significant class\footnote{This, being a subclass of systems with zero mean dimension, contains all the minimal systems (in the subclass). Recall that a system $(X,T)$ is said to be \textbf{minimal} if $X$ is the only $T$-invariant, closed, and nonempty subset of $X$. Indeed, the argument there also applies to those with the \textit{marker property}, encompassing systems admitting infinite minimal factors.} of infinite-dimensional systems behaving well in this property (i.e., any among them has a nontrivial finite entropy factor). On the other hand, Lindenstrauss \cite{Lin95} also showed to the contrary that there do exist some (infinite-dimensional) systems (even minimal) for which this property fails.

For example, we consider the dynamical system $([0,1]^\mathbb{Z},\sigma)$, called \textit{the full shift over the alphabet} $[0,1]$. Recall that the \textbf{full shift} $\sigma$ on the product space $[0,1]^\mathbb{Z}$ is defined by $\sigma:[0,1]^\mathbb{Z}\to[0,1]^\mathbb{Z}$, $(x_n)_{n\in\mathbb{Z}}\mapsto(x_{n+1})_{n\in\mathbb{Z}}$. It is infinite-dimensional, and has infinite topological entropy, but its mean dimension is finite (equal to $1$ \cite{LW00}). Lindenstrauss \cite{Lin95} proved that this system has no nontrivial finite entropy factors. In fact, as shown in \cite{LW00} by Lindenstrauss and Weiss, $([0,1]^\mathbb{Z},\sigma)$ satisfies the property even stronger: having no nontrivial factors with zero mean dimension. Later, Li \cite{Li13} substantially extended the result \cite{LW00} to all sofic groups for all the path-connected alphabets.

Note that in the two conditions stated above, the former is equivalent to saying that every nontrivial factor has infinite topological entropy, while the latter, positive mean dimension. Thus it is tempting to consider systems with all the nontrivial factors having infinite mean dimension. But this simply turns out to be not proper, because it is easy to show that any system has a nontrivial factor (dynamically) isomorphic to a subsystem of $([0,1]^\mathbb{Z},\sigma)$ \cite[Proposition 1.2]{Lin95}, which then implies having finite mean dimension \cite{LW00}. So the correct family for further consideration, as an alternative, should be the systems with the previously-mentioned property, namely, with all the nontrivial factors having positive mean dimension.

Recently, Garcia-Ramos and Gutman \cite{GG24} revisited this notion and provided an abstract approach to viewing it. Although those examples given in the earlier work are only of two types: the full shifts and certain minimal systems being complicatedly constructed, they still systematically investigated the family\footnote{These systems are said there to have \textit{completely positive mean dimension}.} of dynamical systems with this property. Through some local analysis along the technical detail of the proofs given in \cite{LW00,Li13}, Garcia-Ramos and Gutman \cite{GG24} showed that, among other things, this family in some sense is considerably rich.

To a certain extent, we also intend to study the dynamical systems with the property that all the nontrivial factors have positive mean dimension (and thus, have infinite topological entropy). However, we do not plan to go within this framework. As mentioned before, there are few kinds of examples (up to the authors' knowledge) that are known to satisfy this property, but the existence of such systems proved to be rather rich. These two aspects hence give rise to a seemingly unsatisfying impression that stimulates us to enrich the concrete ones. In this paper, we wish it to become clarified. More specifically, our aim is to develop an ``if and only if'' condition for this property among a typical class of dynamical systems, called \textit{the subshifts of block type in the Hilbert cube}. As we can see in a moment, this will lead in particular to a large class of concrete (and new) examples of dynamical systems having this property.

To state our main theorem precisely, we recall some necessary terminology\footnote{For the definition of \textit{mean dimension} introduced by Gromov--Lindenstrauss--Weiss as a new topological invariant (other than the classical \textit{topological entropy}) of dynamical systems, we refer to \cite{LW00}.}. Let $K$ be a compact metrizable space. Let $q$ be a positive integer. Let $B$ be a closed (and nonempty) subset of the product space $K^q$. Define a (closed) subset $X_0$ of the product space $K^\mathbb{Z}$ by $$X_0=\{(x_i)_{i\in\mathbb{Z}}\in K^\mathbb{Z}:(x_i)_{kq\le i\le kq+q-1}\in B,\,\;\forall k\in\mathbb{Z}\}.$$ We consider $(K^\mathbb{Z},\sigma)$, the full shift over the alphabet $K$. Now define the (closed and $\sigma$-invariant) subset $X\subset K^\mathbb{Z}$ as the one consisting of all the points $p$ in $K^\mathbb{Z}$ satisfying that there is some $t\in\mathbb{Z}$ such that $\sigma^t(p)\in X_0$. Clearly, if we consider the restriction of $\sigma$ to $X$ (which is still denoted by $\sigma$), then $(X,\sigma)$ comes to be a subsystem of $(K^\mathbb{Z},\sigma)$. Following Coornaert \cite[Subsection 7.5]{Coo15} we say that $(X,\sigma)$ is \textbf{the subshift of block type of $(K^\mathbb{Z},\sigma)$ associated with the pair $(q,B)$}. The subshifts of block type form a large class of dynamical systems, which are widely used (frequently being a practical and key tool and sometimes with possible refinements) to construct systems for various purposes. These systems and their factors are mostly of high complexity, which can be also reflected in our main result as follows.

\begin{theorem}[Main theorem]\label{main1}
Let $B$ be a path-connected and closed subset of $[0,1]^2$. Let $(X,\sigma)$ be the subshift of block type of $([0,1]^\mathbb{Z},\sigma)$ associated with $(2,B)$. Then the following three conditions are equivalent:
\begin{itemize}
\item[\textup{(i)}] every nontrivial factor of $(X,\sigma)$ has infinite topological entropy;
\item[\textup{(ii)}] every nontrivial factor of $(X,\sigma)$ has positive mean dimension;
\item[\textup{(iii)}] the set $B$ and the diagonal of $[0,1]^2$ intersect, i.e., $B\cap\{(x,x):0\le x\le1\}\ne\emptyset$.
\end{itemize}
\end{theorem}

Note that when removing Condition (ii), Theorem \ref{main1} becomes a statement that does not involve the notion of mean dimension at all, and still establishes an ``if and only if'' condition which is new even for topological entropy of factors. Also note that Condition (iii) in Theorem \ref{main1} is rather elementary and explicit, and as a direct consequence, we see that such systems are indeed not difficult to find (and easy to construct).

\begin{remark}
The path-connectedness of $B$ required in Theorem \ref{main1} is a very natural condition which cannot be removed. See Section 3.
\end{remark}
\begin{remark}
The alphabet $[0,1]$ of the full shift $([0,1]^\mathbb{Z},\sigma)$ in the assumption of Theorem \ref{main1} cannot be replaced with a general compact metrizable space, even the one-dimensional torus. See Section 3 for details.
\end{remark}
\begin{remark}
The length $2$ of the block $B$ concerned in Theorem \ref{main1} cannot be simply replaced by an arbitrary positive integer $n$, even if $n=3$. Section 3 includes several examples in relation to possible generalizations.
\end{remark}

This paper is briefly organized as follows: In Section 2, we prove Theorem \ref{main1}; in particular, we develop a general lemma for that property. In Section 3, we present variations of our main theorem in company with basic examples, and moreover, we have a discussion on some open problems that are worth further exploring.

\section{Proof of the main theorem}
In this section we prove Theorem \ref{main1}. First note that Condition (ii) implies Condition (i), since Lindenstrauss and Weiss \cite[Section 4]{LW00} showed that if a dynamical system has finite topological entropy, then it must have zero mean dimension. Therefore it remains to prove the implications (i)$\implies$(iii) and (iii)$\implies$(ii).

Let $B$ and $X$ be the ones as assumed in the statement of Theorem \ref{main1}. Clearly, $X$ may\footnote{This follows directly from the characterization of the elements of $X$.} be written as a union of two closed subsets of $[0,1]^\mathbb{Z}$: $X=X_0\cup X_1$, where the closed subsets $X_0,X_1\subset[0,1]^\mathbb{Z}$ are defined by $$X_0=\{(x_i)_{i\in\mathbb{Z}}\in[0,1]^\mathbb{Z}:(x_i)_{2k\le i\le2k+1}\in B,\,\;\forall k\in\mathbb{Z}\},$$$$X_1=\{(x_i)_{i\in\mathbb{Z}}\in[0,1]^\mathbb{Z}:(x_i)_{2k-1\le i\le2k}\in B,\,\;\forall k\in\mathbb{Z}\}.$$

\subsection*{Proof of (i)$\implies$(iii)}
Suppose that Condition (iii) is not satisfied, i.e., $$B\cap\{(a,a):0\le a\le1\}=\emptyset.$$ We claim that $X_0\cap X_1=\emptyset$.

In fact, it follows from the path-connectedness of $B\subset[0,1]^2$ that $B$ is contained in either $\{(a_1,a_2):0\le a_1<a_2\le1\}$ or $\{(a_1,a_2):1\ge a_1>a_2\ge0\}$. Without loss of generality we may assume $$B\subset\{(a_1,a_2):0\le a_1<a_2\le1\}.$$ If there is a point $x=(x_i)_{i\in\mathbb{Z}}\in[0,1]^\mathbb{Z}$ such that $x\in X_0\cap X_1$, then from $x\in X_0$ we have $(x_{2k},x_{2k+1})\in B$ which implies that $$0\le x_{2k}<x_{2k+1}\le1,\quad\quad\;\forall\,k\in\mathbb{Z},$$ and from $x\in X_1$ we have $(x_{2k-1},x_{2k})\in B$ which implies that $$0\le x_{2k-1}<x_{2k}\le1,\quad\quad\;\forall\,k\in\mathbb{Z}.$$ Thus, we have $$0\le x_k<x_{k+1}\le1$$ for all $k\in\mathbb{Z}$, and hence, the sequence $(x_k)_{k=1}^{+\infty}$ converges to some $b\in[0,1]$ as $k\to+\infty$. By the compactness of the set $B\subset[0,1]^2$ we have $(b,b)\in B$. This implies that $$B\cap\{(a,a):0\le a\le1\}\ne\emptyset,$$ and therefore contradicts our assumption. So the claim is then proven.

Now by this claim, we are allowed to define a map $$\phi:X_0\cup X_1=X\to P=\{p_0,p_1\}$$ by $\phi(X_i)=p_i$ for $i\in\{0,1\}$, where $p_0$ and $p_1$ are two distinct points. We equip the set $P$ with a map $T:P\to P$ defined by $T(p_i)=p_{1-i}$ for $i\in\{0,1\}$. Note that under the discrete topology the finite set $P$ comes to be a compact metrizable space and $T:P\to P$ a homeomorphism, and in consequence $(P,T)$ becomes a dynamical system. By definition we have $\sigma(X_i)=X_{1-i}$ for $i\in\{0,1\}$. It follows that the map $\phi:X\to P$ is continuous, surjective, and equivariant (namely, it satisfies the commutativity: $\phi\circ\sigma=T\circ\phi$). This shows that $(P,T)$ is a factor of the dynamical system $(X,\sigma)$. Since the system $(P,T)$ is nontrivial and obviously has zero topological entropy, Condition (i) is not fulfilled.

\subsection*{Proof of (iii)$\implies$(ii)}
Now we suppose that Condition (iii) is satisfied. Then we must have $$X_0\cap X_1\ne\emptyset.$$ To see this fact, it suffices to observe that any $\sigma$-fixed point $(a)_{i\in\mathbb{Z}}\in[0,1]^\mathbb{Z}$ with the requirement $(a,a)\in B$ has to lie in the intersection of $X_0$ and $X_1$.

Next we turn to considering the two systems $(X_0,\sigma^2)$ and $(X_1,\sigma^2)$, where $\sigma^2$ is to denote the homeomorphism $\sigma\circ\sigma:X\to X$ (restricted to subsets automatically if necessary). Note that both of them are subsystems of the dynamical system $(X,\sigma^2)$. Here the point is that although $X_0$ and $X_1$ are not $\sigma$-invariant, it is easy to see that they are $\sigma^2$-invariant closed subsets of $X\subset[0,1]^\mathbb{Z}$.

It is clear that both the systems $(X_0,\sigma^2)$ and $(X_1,\sigma^2)$ are dynamically isomorphic (i.e., topologically conjugate) to another full shift $(B^\mathbb{Z},\sigma)$, the full shift over the alphabet $B$. Since $B$ is a path-connected compact metrizable space, by Li's theorem\footnote{We borrow Li's result as it applies to all the path-connected alphabets, which is then more flexible to use.} \cite[Theorem 7.5]{Li13} we obtain that $(B^\mathbb{Z},\sigma)$ satisfies the property that all the nontrivial factors have positive mean dimension (and thus infinite topological entropy), and hence, so do $(X_0,\sigma^2)$ and $(X_1,\sigma^2)$.

Recall that $X=X_0\cup X_1$.

We need develop a general lemma below, Lemma \ref{general}. Applying Lemma \ref{general} to these objects we can deduce immediately that all the nontrivial factors of $(X,\sigma^2)$ have positive mean dimension.

Furthermore, notice that if $(Z,R)$ is any nontrivial factor of $(X,\sigma)$, then $(Z,R^2)$ comes to be a nontrivial factor of $(X,\sigma^2)$. Thus, by Lindenstrauss and Weiss \cite[Proposition 2.7]{LW00} (which states that the mean dimension of $(Z,R)$ is equal to a half of that of $(Z,R^2)$) we conclude that all the nontrivial factors of $(X,\sigma)$ have positive mean dimension, as required by Condition (ii).

\begin{lemma}\label{general}
Let $(Y,S)$ be a dynamical system. Suppose that $(Y_1,S)$ and $(Y_2,S)$ are subsystems of $(Y,S)$ such that $Y=Y_1\cup Y_2$ and $Y_1\cap Y_2\ne\emptyset$. If both $(Y_1,S)$ and $(Y_2,S)$ satisfy the property that all the nontrivial factors have positive mean dimension, then so does $(Y,S)$.
\end{lemma}
\begin{proof}
Take any nontrivial factor $(Z,R)$ of $(Y,S)$ with a factor map $\pi:(Y,S)\to(Z,R)$. Put $Z_1=\pi(Y_1)$ and $Z_2=\pi(Y_2)$. Clearly, $(Z_1,R)$ and $(Z_2,R)$ become factors of $(Y_1,S)$ and $(Y_2,S)$, respectively. Also observe that $(Z_1,R)$ and $(Z_2,R)$ are subsystems of $(Z,R)$, satisfying that $Z=Z_1\cup Z_2$ and $Z_1\cap Z_2\ne\emptyset$. This implies that, if both $Z_1\subset Z$ and $Z_2\subset Z$ contain only one point, then we must have $Z_1=Z_2\subset Z$, and therefore, $Z=Z_1=Z_2$ would have to consist of one point only. It follows that at least one of $(Z_1,R)$ and $(Z_2,R)$ should be nontrivial, and in consequence, this one will have positive mean dimension. So we see that $(Z,R)$ has positive mean dimension, as desired.
\end{proof}

\subsection*{Some remarks}
We have proved Theorem \ref{main1}. Now we wish to give some remarks in relation to Lemma \ref{general}. First note that (the proof directly shows that) the property concerned in the statement of Lemma \ref{general} may be replaced straightforwardly with a property such as that all the nontrivial factors have infinite (or, positive) topological entropy. Besides, it is valid for all the actions of sofic groups. Although this lemma is simple, it is useful in the sense that it enables us to produce new examples with the same property from the old ones. More precisely, such a new system can be constructed by taking (a factor of) a finite union of (factors of) old systems but with, for example, a fixed-point in common, and so on.

Next we point out that in some very special case we actually do not need Lemma \ref{general}, and so the proof could be simplified immediately. To explain it in detail, assume that $B=[c_1,c_1^\prime]\times[c_2,c_2^\prime]\subset[0,1]^2$ is a rectangle that intersects the diagonal of $[0,1]^2$. In this case we have $$X_0=\{(x_i)_{i\in\mathbb{Z}}\in[0,1]^\mathbb{Z}:\,c_1\le x_{2k}\le c_1^\prime,\,c_2\le x_{2k+1}\le c_2^\prime,\,\;\forall k\in\mathbb{Z}\},$$$$X_1=\{(x_i)_{i\in\mathbb{Z}}\in[0,1]^\mathbb{Z}:\,c_1\le x_{2k-1}\le c_1^\prime,\,c_2\le x_{2k}\le c_2^\prime,\,\;\forall k\in\mathbb{Z}\},$$ and hence $(X,\sigma^2)=(X_0\cup X_1,\sigma^2)$ is isomorphic to the full shift on the product space $A^\mathbb{Z}$, for some path-connected alphabet $A\subset[0,1]^2$. Thus, Li's theorem \cite[Theorem 7.5]{Li13} is adequate for the above proof replacing Lemma \ref{general}. The point here is that $x_i$ and $x_{i+1}$ with $(x_i,x_{i+1})\in B$, for each $i\in\mathbb{Z}$, are independent of each other, and in consequence the system $(X,\sigma^2)$, as mentioned above, can be clearly seen and easily described. However, in general the same derivation does not work properly for an arbitrary path-connected closed subset $B\subset[0,1]^2$. This explains why we need this lemma essentially in its general form in order to adapt our proof to all the possible cases.

\section{Open questions}
\subsection{Some examples}
In this subsection we give several examples which exclude some seemingly plausible directions of generalizations of Theorem \ref{main1}. First note that if the block concerned in a subshift is not path-connected, then the case turns out to be less interesting even for the full shifts. For instance, the full shift over the alphabet $[0,1]\cup[2,3]$ admits a factor which is zero-dimensional and which has topological entropy $\log2$.

\subsubsection*{Example regarding general alphabets}
The alphabet $[0,1]$ involved in the assumption of Theorem \ref{main1} cannot be replaced with a general compact metrizable space. For example, consider the one-dimensional torus $\mathbb{T}=[0,1]/\sim$, where the equivalence relation $\sim$ is given by $t\sim t^\prime\iff t-t^\prime\in\mathbb{Z}$ for any $0\le t,t^\prime\le1$ (i.e. with $0$ and $1$ being identified). Let $B=\{(s,s+0.1)\mod1:0\le s\le1\}$. Clearly, $B$ is a path-connected closed subset of $\mathbb{T}^2$, and moreover, $B\cap\{(t,t):t\in\mathbb{T}\}=\emptyset$. Let $(X,\sigma)$ be the subshift of block type of $(\mathbb{T}^\mathbb{Z},\sigma)$ associated with $(2,B)$. As observed, the point $(0.1i\mod1)_{i\in\mathbb{Z}}\in\mathbb{T}^\mathbb{Z}$ belongs to both $X_0$ and $X_1$, and hence, we have $X_0\cap X_1\ne\emptyset$ (where $X_0$ and $X_1$ are similar to the ones as defined in Section 2). Using the same method as exhibited in the proof of Theorem \ref{main1} in Section 2, we can see eventually that every nontrivial factor of $(X,\sigma)$ has positive mean dimension. However, $B$ and the diagonal of $\mathbb{T}^2$ need not intersect. This example shows that it is possible that Condition (ii) does not imply Condition (iii) if\footnote{But notice that the converse (i.e. an analogue of the implication (iii)$\implies$(ii)) is always true (namely, its proof applies to any alphabet and any path-connected block of any length).} we replace $[0,1]$ by a general alphabet.

\subsubsection*{Examples regarding path-connected blocks of arbitrary length}
We include here, without changing notations\footnote{This means that in general (an analogue of) the notation will adapt automatically to the case if necessary (but we shall not mention it). For example, we will similarly have three closed subsets $X_0,X_1,X_2\subset[0,1]^\mathbb{Z}$ generated from the block $B$ with $X=X_0\cup X_1\cup X_2$.}, two examples of subshifts of block type $(X,\sigma)$ of $([0,1]^\mathbb{Z},\sigma)$ associated with $(3,B)$. As follows are some typical diagonal-like closed subsets of the cube $[0,1]^3$: $D_1=\{(c_1,c_2,c_3)\in[0,1]^3:c_2=c_3\}$, $D_2=\{(c_1,c_2,c_3)\in[0,1]^3:c_3=c_1\}$, $D_3=\{(c_1,c_2,c_3)\in[0,1]^3:c_1=c_2\}$, $D=D_1\cap D_2\cap D_3=\{(c,c,c):0\le c\le1\}$, which are related to the block $B$ of length $3$ (contained in $[0,1]^3$), as presented in the\footnote{Condition (iii) focuses mainly on if the block $B$ is disjoint with some analogue of diagonals. Observe that through a similar derivation we can obtain that if Condition (ii) is satisfied, then at least one of the following conditions are fulfilled: $B\cap D_1\ne\emptyset$, $B\cap D_2\ne\emptyset$, $B\cap D_3\ne\emptyset$. So the first example aims to show that the conclusion cannot be strengthened by Condition (iii). Conversely, as already noted, by means of the same method we see that Condition (iii) implies Condition (ii). So the second example aims to show that the assumption cannot be relaxed as stated.} two examples.

Let $B$ be the boundary of the triangle with three vertices $(1,0,0)$, $(0,1,0)$, $(0,0,1)$. Clearly, $B$ is a path-connected and closed subset of $[0,1]^3$, and in addition, $B\cap D=\emptyset$. Note that in this case\footnote{It is also possible to find some $B$ such that $X_0\cap X_1\ne\emptyset$, $X_1\cap X_2\ne\emptyset$, $X_2\cap X_0\ne\emptyset$ (which meets the requirements stated in Lemma \ref{general} as well), but additionally $X_0\cap X_1\cap X_2=\emptyset$ (e.g., this can be done by considering the set $B=\{(l,1-l,0)\in[0,1]^3:0\le l\le1\}$, or, the boundary of the rectangle with vertices $(0,0.5,0)$, $(1,0.5,0)$, $(1,0.5,1)$, $(0,0.5,1)$).} we have $X_0\cap X_1\cap X_2\ne\emptyset$. This will imply\footnote{Roughly speaking, here we still need to follow the proof of Theorem \ref{main1} (and in particular, to employ \cite[Theorem 7.5]{Li13} and Lemma \ref{general} if need be).} that all the nontrivial factors of $(X,\sigma)$ have positive mean dimension. However, the block $B\subset[0,1]^3$ and the set $D\subset[0,1]^3$ are disjoint. This example is to demonstrate that in the assumption of Theorem \ref{main1}, if the length $2$ of a block $B$ is directly replaced by a larger positive integer $3$ while the diagonal of $[0,1]^2$ (concerned in Condition (iii)) is simply replaced by $D$ (the diagonal of $[0,1]^3$), then Condition (iii) may fail (with Condition (ii) being satisfied).

As below we will see\footnote{Furthermore, to clarify something else (e.g., aspects involving diagonals of faces of the cube $[0,1]^3$, or even with the edges involved), one can consider the boundary of the triangle with three vertices $(b,a_1,a_1)$, $(a_2,b,a_2)$, $(a_3,a_3,b)$, provided $b,a_1,a_2,a_3\in[0,1]$ are pairwise distinct.} that conversely, a weaker version of Condition (iii) does not imply Condition (ii). Let $B=\{(l,0.5(1-l),0)\in[0,1]^3:0\le l\le1\}$. It is clear that $B$ is a path-connected and closed subset of $[0,1]^3$, and is such that $B\cap D_i\ne\emptyset$ for $i\in\{1,2,3\}$. Moreover, a direct observation shows that the sets $X_0$, $X_1$, $X_2$ are pairwise disjoint. It follows that there is a nontrivial factor of $(X,\sigma)$, similar to the one given in the proof of Theorem \ref{main1}, which has zero topological entropy. Therefore Condition (ii) is not satisfied. This example is to demonstrate that Condition (ii) may still fail even if all the following conditions are fulfilled: $B\cap D_1\ne\emptyset$, $B\cap D_2\ne\emptyset$, $B\cap D_3\ne\emptyset$.

The variety of examples of this kind is endless. To summarize, these examples show that we cannot expect Theorem \ref{main1} to be generalized essentially with a simple modification only. There is still some possibility that Theorem \ref{main1} could be extended to some level with a new perspective or mixed conditions. The authors leave this paper at its notationally readable and elementary level.

\subsection{Variants of the main result}
If we do not seek a characterization as explicit as the type of Condition (iii) in the statement of Theorem \ref{main1}, then we have the following theorem for all the subshifts of block type.
\begin{theorem}\label{main2}
Let $A$ be a compact metrizable space and $n$ a positive integer. Let $B$ be a path-connected closed subset of $A^n$. Let $(X,\sigma)$ be the subshift of block type of $(A^\mathbb{Z},\sigma)$ associated with $(n,B)$. Then every nontrivial factor of $(X,\sigma)$ has infinite topological entropy if and only if every nontrivial factor of $(X,\sigma)$ has positive mean dimension.
\end{theorem}
We do not regard Theorem \ref{main2} as a main result of this paper (although it holds true for all the subshifts of block type), because its statement does not contain a ``block-only'' condition (which is of a simple form) as clear as the one given in Theorem \ref{main1}. We do not give the proof of Theorem \ref{main2} as it is somewhat similar\footnote{Note that this is actually much easier and more direct to show.} in part to the proof of Theorem \ref{main1}.

The next result (Theorem \ref{main3}) pursues a similar dichotomy for\footnote{The alphabet concerned in a full shift is not so important to Theorems \ref{main2} and \ref{main3}.} each factor of any subshift of block type. The proof of Theorem \ref{main3} is omitted as well. Note that Theorem \ref{main3} implies Theorem \ref{main2}, moreover, Theorem \ref{main3} can be regarded as a statement of topological entropy, too.
\begin{theorem}\label{main3}
Let $A$ be a compact metrizable space and $n$ a positive integer. Let $B$ be a path-connected closed subset of $A^n$. Let $(X,\sigma)$ be the subshift of block type of $(A^\mathbb{Z},\sigma)$ associated with $(n,B)$. Then each factor of $(X,\sigma)$ has either zero topological entropy or positive mean dimension \textup(and hence infinite topological entropy\textup).
\end{theorem}

\subsection{Problems}
The paper ends with a conjecture which is closely related to the current topic. As mentioned previously, a dynamical system having positive mean dimension must have infinite topological entropy \cite[Section 4]{LW00}, but not vice versa. However, the following conjecture seems reasonable. Before stating it, we should indicate that it is shaded by the Gutman--Lindenstrauss--Tsukamoto conjecture\footnote{Gutman--Lindenstrauss--Tsukamoto \cite{GLT16} conjectured that a system has zero mean dimension if and only if it is an inverse limit of finite entropy systems.} for zero mean dimension \cite[Conjecture 1.1]{GLT16}, a more difficult problem about the structure of zero mean dimensional systems.

\begin{conjecture}\label{conjecture}
Let $(X,T)$ be a dynamical system. Then the following conditions are equivalent:
\begin{itemize}
\item every nontrivial factor of $(X,T)$ has infinite topological entropy,
\item every nontrivial factor of $(X,T)$ has positive mean dimension.
\end{itemize}
\end{conjecture}

Note that Conjecture \ref{conjecture} has now been confirmed true for (at least) three classes of dynamical systems as follows: (1) all the minimal systems (having no periodic points, unless the state space is finite) \cite{Lin99,GLT16}, (2) finite-dimensional systems (having mean dimension zero) \cite{Lin95}, and, (3) subshifts of block type (with periodic points involved), verified in the present paper.

Lastly, we would like to remark that Conjecture \ref{conjecture} has not yet been (although the Gutman--Lindenstrauss--Tsukamoto zero mean dimension conjecture \cite[Conjecture 1.1]{GLT16} has been) confirmed true for systems with the marker property\footnote{In particular, any system with an infinite minimal system as a factor has the marker property. The \textit{marker property} is intimately connected with the \textit{small boundary property}. For a conjecture about the relation between the marker property and the small boundary property we refer to \cite[Conjectures 7.4 and 7.5]{TTY22}.}. The point is that any factor of a minimal system is also minimal, while the marker property does not apply (i.e., a factor of a system satisfying the marker property is not necessarily again to possess the marker property), nor do systems admitting infinite minimal factors.

\end{document}